\documentclass[11pt]{amsart}

\usepackage[latin1]{inputenc}
\usepackage{graphicx}
\usepackage{amsmath}
\usepackage{amssymb}
\usepackage[all]{xy}
\usepackage{amsmath,amsthm,amsfonts,amssymb,amscd}
\usepackage{enumerate}
\usepackage{amssymb,amsthm,amsmath,psfrag,mathrsfs}
\usepackage[latin1]{inputenc}
\usepackage{graphicx}

\newtheorem{Theorem}{Theorem}[section]
\newtheorem{Corollary}[Theorem]{Corollary}
\newtheorem{Proposition}[Theorem]{Proposition}
\newtheorem{Lemma}[Theorem]{Lemma}
\theoremstyle{definition}
\newtheorem{Definition}[Theorem]{Definition}
\newtheorem{Remark}[Theorem]{Remark}
 \newtheorem{Example}[Theorem]{Example}
\def\leaderfill{\leaders\hbox to .8em{\hss .\hss}\hfill}
\def\_#1{{\lower 0.7ex\hbox{}}_{#1}}

\def\RR{{\mathcal{R}}}

\def\P{{\mathbb{P}}}
\def\C{{\mathbb{C}}}

\def\O{{\mathcal{O}}}

\def\FF{\mathcal{F}}
\def\LL{{\mathcal{L}}}

\def\sing{\operatorname{{sing}}}

\title[On first order deformations]
{On First order deformations of homogeneous foliations}
\author{Ariel Molinuevo and Bruno Scárdua}

\begin{document}

\maketitle

\begin{abstract}
We study  analytic deformations of holomorphic foliations given by
homogeneous integrable one-forms in the complex affine space
$\mathbb C^n$.  The deformation is supposed to be of first order
(order one in the parameter). We also assume that the deformation is
given by homogeneous polynomial one-forms. The deformation takes
place in the affine space since we are not assuming that the
foliations descent to the projective space. We describe the space of
such deformations in three main situations: (1) the given foliation
is given by the level hypersurfaces of a homogeneous polynomial. (2)
the foliation is rational, ie., has a first integral of type
$P^r/Q^s$ for some homogeneous polynomials $P,Q$. (3) the foliation
is logarithmic of a generic type.
 We prove that, for each class above,  the first order homogeneous deformations of same
degree are in the very same class. We also investigate the existence
of such deformations with different degree.
\end{abstract}
\tableofcontents

\section{Introduction and main results}
\label{section:introduction} Foliations are an important tool in the
classification of  manifolds, specially in low dimension. This
refers initially to the study of smooth non-singular integrable
structures on closed real manifolds. This is particularly evident in
the case of codimension one foliations. Following this spirit the
notion of holomorphic foliation with singularities was brought to
the scene. The subject has grown and spread to other areas as
algebraic geometry. Indeed, the classification of compact  complex
surfaces is strongly related to the study of foliations in such
objects. Much more has been done in this direction. Related to this
is the structure of the space of foliations.  As it is known, the
space of codimension one holomorphic foliations in a complex
manifold has the structure of an analytic variety. One may then ask
for its irreducible components. This is a quite vast and rich topic.
We will focus on a very specific case in this framework. More
precisely, we will study codimension one holomorphic foliations in
the complex affine space of dimension $\geq 3$. We shall restrict
our study to the algebraic (polynomial) case. More precisely, in
this paper we are concerned with the space of deformations of a
polynomial homogeneous one-form satisfying the integrability
condition of Frobenius. We consider deformations as perturbations of
first order, also satisfying the integrability condition. Our
starting one-form is assumed to admit an integrating factor which is
reduced. This implies that the corresponding foliation is
logarithmic in the sense of \cite{omegar} and \cite{Cerveau-Mattei},
\emph{i.e.}, given by a simple poles rational one-form. In a certain
sense after those admitting rational first integral, this is the
simplest class of foliations and correspond to a linear model. A
number of authors have addressed the problem of finding the
irreducible components of the space of codimension one holomorphic
foliations in the complex projective space of dimension $n \geq 3$.
This is treated by studying deformations by same degree homogeneous
integrable one-forms of a given homogeneous integrable one-form
$\omega$ in $n+1$ complex variables $(x_1,...,x_{n+1})$. Since the
corresponding foliation in $\mathbb C^{n+1}\setminus \{0\}$ descends
to the projective space $\mathbb P^n$, we also have $i_R(\omega)=0$
where $R=\sum\limits_{i=1}^{n+1} x_j \frac{\partial}{\partial x_j}$
is the radial vector field.

One of the very first results in this subject is the finding of the
so called \emph{rational components} (\cite{gmln}). This
corresponds to the stability of foliations given by the fibers of
maps of the form $P^r/Q^s\colon \mathbb P^n \to \mathbb P^1$ for $n
\geq 3$ and suitable homogeneous polynomials $P,Q$ where $r =
\partial (Q)$ and $s= \partial (P)$, and where
we are denoting with the symbol $\partial$ the degree of the given
polynomial. The stability of such foliations is pretty much a
consequence of the very special geometry of the projective space
$\mathbb P^n$ (characteristic classes of line bundles)
(\cite{omegar}). In this paper we shall resume the study of
deformations of rational foliations, but for a wider class. Indeed,
we shall consider foliations with a rational homogeneous first
integral, but not necessarily descending to the projective space. We
shall refer to these as \emph{homogeneous affine rational
foliations}. Our main result for this class of  foliations is (cf.
Theorem~\ref{prop0}):

\begin{Theorem}
Let $\omega_{\RR_0} = r\ f_1df_2\ -\ s\ f_2df_2$ define a
homogeneous affine rational foliation of generic type in $\mathbb
C^n, \,n \geq 3$, where $f_1, f_2$ are homogeneous polynomials and
$r, s \in \mathbb C$. All  first order deformations of
$\omega_{\RR_0}$ by homogeneous integrable one-forms of same degree,
are obtained by perturbations
 of the polynomial parameters $f_1$ and $f_2$
 or of the eigenvalues $(r,s)$.
\end{Theorem}

The study of deformations of logarithmic foliations has started with
the work of Calvo Andrade in \cite{omegar} where the author proves
their stability under some mild conditions, \emph{i.e.}, for generic
elements. This result is based on a byproduct of a result of Hirsch
about fixed points for central hyperbolic elements in a group of
diffeomorphisms and a result of Nori for the fundamental group of
the complement of a codimension one divisor with normal crossings in
$\mathbb P^n$. Later on, other authors added more information to
this subject and obtained more general results by considering more
Algebraic geometry type arguments (\cite{csv} and \cite{cgm}). In
this paper we resume this subject with a slightly different
standpoint. We consider perturbations of order one, \emph{i.e.},  of
the form $\omega_t = \omega + t \eta$ where $t$ is a complex
parameter such that $t^2=0$. We consider though the case where
$\omega$ is an \emph{affine} logarithmic foliation. In short, this
means that $\omega$ admits a reduced integrating factor but not
necessarily descends to the projective space, \emph{i.e.},
$i_R(\omega) \not \equiv 0$. This situation has to be dealt with via
different techniques resembling more to the local cases considered
in \cite{Ce-Sc}. Indeed, one rapidly reaches the connections with
the relative cohomology introduced in \cite{CB}. Nevertheless, we
are working with the meromorphic case, so we cannot apply the
results in \cite{CB}. Roughly speaking we shall be concerned with
the following equation
\begin{equation}\label{equ11}
 \omega\wedge d\eta + d\omega\wedge \eta=0,
\end{equation}
which parameterize, in $\eta$, the perturbations of order one of the
given foliation $\omega$, see Section (\ref{section-3}).

\

Let us make our framework more clear. We are considering one-forms
 $\omega$ as follows:
 \[
\omega = \left(\prod_{i=1}^s f_i \right)\left(\sum_{i=1}^s \lambda_i
\frac{df_k}{f_k}\right)
 \]
where $\lambda_i \in \mathbb C$ are called \emph{eigenvalues of
$\omega$} and the  homogeneous polynomials $f_1,\ldots,f_s$ are the
\emph{polynomial parameters} of $\omega$. We shall refer to
$\omega$ then as a \emph{homogeneous affine logarithmic foliation}
in $\mathbb C^n$. We shall say that $\omega$ is \emph{generic} if
it verifies the following conditions (\cite{cgm}):
  \begin{enumerate}
   \item the $\{f_i=0\}$ are smooth, irreducible $\forall i=1,\ldots, s$
   and $D=\{f_1.\ldots .f_s=0\}$ is a divisor with normal crossings
   \item $\lambda_i \neq \lambda_j (\neq 0)$ for every $i\neq j$.
  \end{enumerate}

 We give solutions to the eq. (\ref{equ11}) above in all cases. We shall also describe all the
solutions of same degree of $\omega$ for a generic element $\omega$.
As a result we are able to prove (see Theorem~\ref{prop1} for a
complete statement) for dimension $n \geq 3$:

\begin{Theorem}
All same degree polynomial first order deformations of a generic
homogeneous affine logarithmic foliation, defined by integrable
homogeneous one-forms of same degree, are obtained by perturbations
 of the polynomial parameters $f_i$ or of the eigenvalues
 $\lambda_i$.
 \end{Theorem}

Finally, in the last part of this work, we consider first order
integrable perturbations of an exact homogeneous one-form
$\omega=dP$. This may be considered as a global version of the main
result in \cite{Ce-Sc}. We obtain  a particular case of a result
recently proved in \cite{cs}.

\begin{Theorem}\label{teo-exactintro}
 Let $\omega=dP$ be an exact differential form
 in $\Omega^1_{\C^n}, \, n \geq 3$, homogeneous of degree $e$. Let us suppose
 also that the codimension of the singular locus of $dP$ is $\geq 3$.
 Then all first order deformations of $\omega$, defined by integrable homogeneous one-forms
 of the same degree, are exact of type
 \[
\omega_\varepsilon=d(P + \varepsilon Q)
 \]
where $Q$ is a homogeneous polynomial of degree $e$.
\end{Theorem}

Theorem~\ref{teo-exactintro} is proved as Theorem~\ref{teo-exact}
and may be seen as a global version of Malgrange's ``singular
Frobenius" result (\cite{malgrangeI}).

\section{Notation}

We will denote with $S_{n}=\C[x_1,\ldots,x_n]$ the ring of
polynomial in $n$ complex variables. We would like to recall here
that the global sections of the twisted sheaf of differential forms
of $\Omega^1_{\P^n}(e)$ it is a finitely generated $S_{n+1}$-module
defined by $\omega\in H^0\left(\Omega^1_{\P^n}(e)\right)$ if and
only if $\omega$ can be written as
\[
  \omega = \sum_{i=1}^{n+1} A_i dx_i
\]
where:
\begin{enumerate}
\item The $A_i$'s are homogeneous polyonomials in $S_{n+1}$, of degree $e-1$.
\item The one-form  $\omega$ satisfies the condition of descent to the projective
space: given the radial vector field $R = \sum_{i=1}^{n+1}
x_i\frac{\partial}{\partial x_i}$ we have
\[
 i_R(\omega) = \sum_{i=1}^{n+1} x_i A_i = 0\ .
\]

\end{enumerate}

Regarding foliations in the affine space $\C^n$, we are going to
denote as $\Omega^1_{\C^n}$ the $S_n$-module of K\"ahler
differentials. This module it is given by polynomial 1-differential
forms in $n$-variables, \emph{i.e.}, it is generated by the
differentials $\left(dx_1,\ldots,dx_n\right)$.

\

If $\eta\in\Omega^1_{\C^n}$ we will say that $\eta$ is \emph{
homogeneous of degree $e$} if $\eta$ is of the form
\[
 \eta = \sum_{i=1}^n H_i dx_i
\]
where the $H_i$ are homogeneous polynomials of degree $e-1$. We will
also denote as $\partial(\eta)$ or $\partial(H_i)$ to the degree of
$\eta$ and $H_i$, respectively.

\

We also by $\Omega^1_{\C^n,0}$ the space of germs of differential
forms, \emph{i.e.}, an element $\omega\in\Omega^1_{\C^n,0}$  is of
the form
\[
 \omega = \sum_{i=1}^n \widetilde{A}_i dx_i
\]
where the $\widetilde{A}\in\O_{\C^n,0}$ are germs of holomorphic
functions at the origin.

\

\section{First order deformations of a codimension one foliation}\label{section-3}

\ As it is well-known any holomorphic foliation of codimension one
with singularities in the complex projective space $\mathbb P^n$ is
given by a section $\omega\in H^0(\Omega^1_{\P^n}(e))\big/ \C$ , for
example see \cite{fji}, for some $e\in\mathbb N$.
 We will therefore denote as $\FF^1(\P^n,e)$ the space of \emph{codimension one
foliations} in $\P^n$ of degree $e-2$, \emph{i.e.},
\[
 \FF^1(\P^n,e) = \{\omega\in H^0(\Omega^1_{\P^n}(e))\big/ \C\ :\ \omega\wedge d\omega = 0 \} \ .
\]

\

The \emph{first order deformations} of an integrable differential
one form $\omega\in H^0(\Omega^1_{\P^n}(e))$, are given by the
$\eta\in H^0(\Omega^1_{\P^n}(e))$ such that
\[
 \omega_{\varepsilon} = \omega + \varepsilon \eta
\]
is integrable (in the sense of Frobenius), where the parameter
$\varepsilon$ is infinitesimal, in the sense that $\varepsilon^2=0$.
The integrability condition means that $\omega_\varepsilon\wedge d
\omega_\varepsilon = 0$.  If we expand this equation we get
$\omega_{\varepsilon}\wedge d\omega_{\varepsilon} =
  {\omega\wedge d\omega} +
  \varepsilon \left(\omega\wedge d\eta + d\omega\wedge \eta\right)=0$.
  Since $\omega$ is already integrable,
  the integrability condition of $\omega_\varepsilon$ is then
equivalent to the following equation:
\begin{equation}\label{equ0}
\omega\wedge d\eta + d\omega\wedge \eta=0.
\end{equation}

We may therefore parameterize first order deformations of
$\omega\in\FF^1(\P^n,e)$ as the space
\[
 \begin{aligned}
D_{\P^n}(\omega) &= \left\{\eta\in H^0\left(\Omega^1_{\P^n}(e)\right):\ \omega\wedge d\eta + d\omega\wedge \eta = 0\right\}\big/\C.\omega \\
\end{aligned}
\]

As expected, the vector space $D(\omega)$ can be identified with the
\emph{tangent space} at $\omega$ $T_{\omega}\FF^1(\P^n,e)$, see
\cite[Section 2.1, pp.~709]{fji}.

\

In the affine (algebraic) case, let us define the space of
\emph{affine codimension one foliations}\footnote{We point-out that
there are codimension one holomorphic foliations with singularities
in the affine space $\mathbb C^n$ which are not given by polynomial
one-forms. We shall be working with those which are given by
polynomial one-forms.} as
\[
 \FF^1(\C^n) = \{\omega\in \Omega^1_{\C^n}\big/ \C\ :\ \omega\wedge d\omega = 0 \} \ ,
\]

where we are considering $\omega\in\FF^1(\C^n)$ homogeneous of
degree $e$. \

Similarly to above  the space of first order perturbations of
$\omega$ is defined as
\[
 D_{\C^n}(\omega) = \{\eta\in \Omega^1_{\C^n}\ : \ \omega\wedge d\eta + d\omega\wedge \eta = 0 \}\big/ \C.\omega\ .
\]

\

We shall only consider deformations preserving the degree of the
given foliation. Thus, given $\omega$ of degree $e$  we define the
space of deformations homogeneous of the same degree of $\omega$ as
\begin{equation}\label{def}
  D_{\C^n}(\omega,e) = \{\eta\in\Omega^1_{\C^n}\ : \ \omega\wedge d\eta + d\omega\wedge \eta = 0, \text{ and }\eta\text{ of degree }e\}\big/ \C.\omega\ .
\end{equation}

\

\section{Rational and logarithmic foliations in $\mathbb P^n$}

\

Very basic examples of foliations in $\mathbb P^n$ are given by the
classes of rational and logarithmic foliations. In this section we
should review their definitions and basic results that we are going
to use in the rest of the paper.

\subsection{Rational foliations}
In the case of foliations in the projective space, the class of
rational foliations corresponds to the pull-backs of the two
dimensional model $xdy - ydx=0$ by maps $\sigma: \mathbb P^n \to
\mathbb P^2$ of the form $\sigma=(P^r, Q^s)$ where $P,Q$ are
homogeneous polynomials in $\mathbb C^{n+1}$ of degree $\partial (P)
=s$ and $\partial (Q)=r$. More precisely we have:
\begin{Definition} A \emph{rational foliation} of type
$(d_1,d_2)$ in $\FF^1(\P^n,e)$, is defined by a global section
$\omega_\RR\in H^0\left(\Omega^1_{\P^n}(e)\right)$ of the form
\begin{equation*}
\omega_{\RR} = d_1f_1df_2- d_2f_2df_1,
\end{equation*}
where $\partial(f_1) = d_1$ and $\partial (f_2) = d_2$ and
$d_1+d_2=e$.
\end{Definition}

In the definition above, the $-$ sign and the coefficients $d_1$ and
$d_2$ are taken in order to guarantee the descent to projective
space of the differential form $\omega_{\RR}$.

\

We will note $\RR(n,(d_1,d_2))$ the \emph{space of rational
foliations} of this kind, and define the \emph{generic} open set
$\mathcal{U}_\RR\subset \RR(n,(d_1,d_2))$ as
\begin{equation}
\mathcal{U}_\RR = \{\omega\in \RR(n,(d_1,d_2)):\
codim(Sing(d\omega))\geq 3,\ codim(Sing(\omega))\geq 2 \}.
\end{equation}

\

First order deformations of rational foliations are studied in the
works \cite{gmln} and \cite{fji}. We recall from \cite[Proposition
2.4, p.~711]{fji} the following result.
\begin{Theorem}\label{teo-rat} Let $\omega_{\RR}\in\mathcal{U}_{\RR}$ be a generic rational foliation. Then, the first order deformations of $\omega_{\RR}$, or the tangent space of $\FF^1(\P^n,e)$ at $\omega_\RR$, can be given by the perturbations of the parameters $f_1$ and $f_2$
\[
\begin{aligned}
T_{\omega_\RR}\FF^1(\P^n,e) = D_{\P^n}(\omega_\RR) &= Span\left(\left\{ \eta\in\RR(n,(d_1,d_2)):\eta = d_1f_1'df_2- d_2f_2df_1' \text{ or }\right.\right.\\
&\hspace{4.95cm}\left.\left.\eta = d_1f_1df_2'- d_2f_2'df_1
\right\}\right)\big/ \C.\omega_{\RR}.
\end{aligned}
\]
\end{Theorem}

\subsection{Logarithmic foliations}
Logarithmic foliations in the projective space $\mathbb P^n$ are
pull-back of linear foliations in $\mathbb C^s$, of the form
$\left(\prod\limits_{i=1}^s x_i\right)\sum\limits_{i=1}^s\lambda_i
\frac{dx_i}{x_i}=0$ by maps $\sigma \colon \mathbb P^n \to \mathbb
P^s$. In a more formal way we have:
\begin{Definition}\label{condition1}
 A \emph{logarithmic foliation} of type $(d_1,\ldots,d_s)$  in $\FF^1(\P^n,e)$, is defined
 by a global section  $\omega_\LL\in H^0(\Omega^1_{\P^n}(e))$ of the form
\begin{equation}\label{logarithmic}
\omega_{\LL} = \left(\prod_{i=1}^s f_i \right)\sum_{i=1}^s\lambda_i
\frac{df_i}{f_i},
\end{equation}
where $s\geq 3$ and
\begin{enumerate}
\item $(\lambda_1,\ldots,\lambda_s)\in\Lambda(s) :=\{(\lambda_1,\ldots,\lambda_s)\in\C^s:\  \lambda_1d_1+\ldots+\lambda_s d_s=0\}$
\item $f_i$ is homogeneous of degree $d_i$ and $d_1+\ldots+d_s=e$.
\end{enumerate}
\end{Definition}

We will note $\LL(\P^n,\overline{d})$ the \emph{space of
logarithmic foliations} of this kind and define the \emph{generic}
open set $\mathcal{U}_\LL(d)\subset \LL(\P^n,\overline{d})$ as
\begin{equation}\label{gen-log}
\mathcal{U}_\LL (d):= \left\{\omega\in \LL(\P^n,(\overline{d})):\
\omega\text{ verifies a) and b) below }\right\},
\end{equation}
writing $\omega=\left(\prod_{i=1}^s f_i \right)\sum_{i=1}^s\lambda_i
\frac{df_i}{f_i}$ we have the conditions:
\begin{enumerate}
\item[a)] the $\{f_i=0\}$ are smooth, irreducible $\forall i=1,\ldots, s$
   and $D=\{f_1.\ldots .f_s=0\}$ is a divisor with normal crossings
\item[b)] $\lambda_i \neq \lambda_j (\neq 0)$ for every $i\neq j$.
\end{enumerate}

\ Then ${\mathcal U}_\LL(d)$ is a Zariski dense open subset of
$\LL(\P^n,(\overline{d}))$.  We will usually note $\overline{d}$,
$\overline{\lambda}$ and $\overline{f}$ the $s$-uples involved in
the expression of a logarithmic foliation. Noting $F_i =
\prod_{j\neq i} f_j$, we will frequently write $\omega_{\LL}$ as
\begin{equation}\label{logarithmic2}
\omega_{\LL} = \sum_{i=1}^s \lambda_i\  F_i \ df_i.
\end{equation}

\

We now give  an example:
\begin{Example}
{\rm Given homogeneous polynomials $P_1, P_2, Q$ of same degree in
$n$ complex variables, we consider
\[
\omega_\varepsilon =(P_1 + \varepsilon Q)P_2 \big[\lambda_1
\frac{d(P_1 + \varepsilon Q)}{P_1 + \varepsilon Q} +\lambda_2
\frac{dP_2}{P_2}\big]
\]
Then $\omega_\varepsilon = \omega + \varepsilon \eta$ where $\omega=
\lambda_1 P_2 dP_1 + \lambda _2 P_1 dP_2$ and $\eta=\lambda_1 P_2 dQ
+ \lambda_2 Q dP_2$.  Thus we have a first order deformation of a
logarithmic foliation. We may choose the eigenvalues $\lambda_1,
\lambda_2$ and polynomial parameters $P_1, P_2$ in such a way that
the logarithmic form associated to  $\omega$ is generic. The
deformation $\omega_\varepsilon$ is given by logarithmic one-forms,
obtained by first order perturbation in the polynomial parameter
$P_1$.  }
\end{Example}

\

Let us fix $\omega_{\LL}\in\LL(\P^n,\overline{d})$ as before, as in
eq. (\ref{logarithmic2}), and define the \emph{spaces of
perturbation of parameters} of $\omega_{\LL}$ as
\[
\begin{aligned}
D_{\P^n}(\omega_\LL,\overline{f}) &= Span\left(\{\eta_{g_i}
\in\LL(\P^n,\overline{d}):\ \eta_{g_i}\text{ equals }\omega_{\LL}
\text{ with }f_i\text{ replaced by }g_i\}\right)\big/\C.\omega_\LL\\
D_{\P^n}(\omega_\LL,\overline{\lambda}) &=
Span\left(\{\eta_{\overline{\mu}}\in\LL(\P^n,\overline{d}):\
\eta_{\overline{\mu}}\text{ equals }\omega_{\LL}\text{ with
}\overline{\lambda}\text{ replaced by
}\overline{\mu}\}\right)\big/\C.\omega_\LL.
\end{aligned}
\]
By direct computation, it is straight forward to check that
$D_{\P^n}(\omega_\LL,\overline{f})$ and
$D_{\P^n}(\omega_\LL,\overline{\lambda})$ are subspaces of
$D_{\P^n}(\omega_{\LL})$.

\

We know from  \cite[Theorem 25, pp.~14, and Remark 26, pp.~15]{cgm}
 that  the tangent space of $\FF^1(\P^n,e)$ at a generic point given by
$\omega_{\LL}$ it is defined by these perturbations:
\begin{Theorem}Let $\omega_{\LL}\in \mathcal{U}_\LL(d)\subset
\LL(\P^n,\overline{d})$ be a generic logarithmic foliation. Then,
the first order deformations of $\omega_{\LL}$, or the tangent space
of $\FF^1(\P^n,e)$ at $\omega_\LL$, can be decomposed as
\[
T_{\omega_{\LL}}\FF^1(\P^n,e) = D_{\P^n}(\omega_{\LL}) =
D_{\P^n}(\omega_\LL,\overline{f}) \oplus
D_{\P^n}(\omega_\LL,\overline{\lambda})\ .
\]
\end{Theorem}

\section{Rational and logarithmic foliations in $\mathbb C^n$}

We can repeat the definitions of the preceeding section for rational
and logarithmic foliations defined in the affine space $\C^n$.

In this situation, what changes is  condition (1) in definition
(\ref{condition1}) of logarithmic foliation, \emph{i.e.}, we do not
need anymore that
\[
 \sum_{i=1}^s \lambda_i d_i = 0\ .
\]
The same goes for rational foliations, now we can take every pair of
coefficients $(r,s)$ in the definition of rational foliation.

\subsection{Affine rational foliations defined by homogeneous one-forms}
We shall now introduce an intermediate class between the class of
rational foliations in the projective space and the class of
foliations admitting a rational first integral in the affine space.
For this sake we shall consider latter foliations which are defined
by homogeneous one-forms. Indeed, we also allow some more generic
models since we do not demand the eigenvalues to have a rational
quotient.  We define:
\begin{Definition} \label{rat-affine}A \emph{homogeneous affine rational foliation} of type $(d_1,d_2)$ and degree $e$ in $\FF^1(\C^n)$, is defined
by an element $\omega_{\RR_0}\in\Omega^1_{\C^n}$ of the form
 \begin{equation}
  \omega_{\RR_0} = r \ f_1 df_2 \ - \ s\  f_2 df_1\ ,
 \end{equation}
 where $f_1$ and $f_2$ are homogeneous of degree $d_1$ and $d_2$ respectively, $d_1+d_2=e$ and $r,s\in\C$. We will denote with $\overline{f}$ the pair $(f_1,f_2)$ and with $\overline{d}$ the pair $(d_1,d_2)$.
\end{Definition}
We shall refer to $f_1, f_2$ as the \emph{polynomial parameters}
and to $r,s$ as the \emph{eigenvalues} of the foliation. If no
confusion can arise we will call this foliations just \emph{affine
rational foliations} or \emph{rational foliations} as well.

\

We will note $\RR(\C^n,\overline{d})$ the \emph{space of affine
rational foliations} of this kind and define the \emph{generic}
open set $\mathcal{U}_{\RR_0}\subset \RR(\C^n,\overline{d})$ as
\begin{equation}\label{gen-rat2}
\mathcal{U}_{\RR_0} = \left\{\omega_{\RR_0}\in
\RR(\C^n,\overline{d}):\ \omega_{\RR_0}\text{ verifies a) and b)
below }\right\},
\end{equation}
writing $\omega_{\RR_0} = r f_1 df_2 - s f_2 df_1$ we have the
conditions:
\begin{enumerate}
\item[a)] $D = \{f_1.f_2= 0\}$ is a normal crossing divisor
\item[b)] $r\neq -s (\neq 0)$.
\end{enumerate}

\begin{Remark}
{\rm We stress the fact that, the eigenvalues $r,s$ are allowed to
be with non-rational quotient. Thus, our definition above of
rational foliation in the affine space, includes  the linear
hyperbolic case $ xdy - \lambda ydx=0, \lambda \in \mathbb C
\setminus \mathbb R$ as well.

}
\end{Remark}

Let us now consider $\omega_{\RR_0}\in\RR(\C^n,\overline{d})$ of the
form of Definition (\ref{rat-affine}), then we define the subspaces
of $D_{\C^n}(\omega_{\RR_0},e)$ as
\[
\begin{aligned}
D_{\C^n}(\omega_{\RR_0},\overline{f}) &= Span\left(\{\eta\in\RR(\C^n,\overline{d}):\ \eta = rf_1'df_2- sf_2df_1' \text{ or }\right.\\
&\hspace{5.35cm}\left.\eta = rf_1df_2'- sf_2'df_1 \}\right)\big/\C.\omega_{\LL_0}\\
D_{\C^n}(\omega_{\RR_0},(r,s)) &=
Span\left(\{\eta_{(r',s')}\in\RR(\C^n,\overline{d}):
\eta_{(r',s')}=r'f_1df_2-s'f_2df_1\}\right)\big/\C.\omega_{\LL_0}
\end{aligned}
\]

\

Later, in Theorem (\ref{prop1}), we will see that these two spaces
span all the first order deformations of the same degree of
$\omega_{\RR_0}$. Notice that there is no equivalent space of
$D_{\C^n}(\omega_{\RR_0},(r,s))$ in the projective deformations
case, this is because the condition of descent to projective space
forces the coefficients to be such that there is no possible
perturbations of them.

\

\begin{Remark}\label{rem-0}We would like to notice that we
 may also consider  the space $D_{\C^n}(\omega_{\RR_0}, \overline{f})^+ $
 defined as
\[
\begin{aligned}
  D_{\C^n}(\omega_{\RR_0},\overline{f})^+ &= Span\left(\{\eta\in\RR(\C^n,\overline{d'}):\ \eta = rf_1'df_2- sf_2df_1' \text{ or }\right.\\
&\hspace{5.35cm}\left.\eta = rf_1df_2'- sf_2'df_1 \}\right)\big/\C.\omega_{\LL_0}\\
\end{aligned}
\]
where $d_i'$ is the degree of the polyonomial $f_i$ and/or $f_i'$,
which can be different from the original degree $d_i$.
\end{Remark}

\

By direct computation, it is straight forward to check that
$D_{\C^n}(\omega_{\RR_0},\overline{f})$ and
$D_{\C^n}(\omega_{\RR_0},$ $(r,s))$ are subspaces of
$D_{\C^n}(\omega_{\RR_0},e)$ and that the space
$D_{\C^n}(\omega_{\RR_0},\overline{f}) ^+$ is a subspace of
$D_{\C^n}(\omega_{\RR_0})$.

\subsection{Affine logarithmic foliations defined by homogeneous one-forms}

We shall now consider foliations of logarithmic type, but which are
defined by homogeneous one-forms, though not necessarily satisfying
the condition to descent  to the projective space.

\begin{Definition}\label{log-affine}
 A \emph{homogeneous affine logarithmic foliation} of type $(d_1,\ldots,d_s)$
 and degree $e$ in $\FF^1(\C^n)$, is defined by an element
 $\omega_{\LL_0}\in \Omega^1_{\C^n}$ of the form
\begin{equation}\label{logarithmic3}
\omega_{\LL_0} = \left(\prod_{i=1}^s f_i
\right)\sum_{i=1}^s\lambda_i \frac{df_i}{f_i}=\sum_{i=1}^s
\lambda_i\  F_i \ df_i,
\end{equation}
where $s\geq 3$ and
\begin{enumerate}
\item $f_i$ is homogeneous of degree $d_i$ and $d_1+\ldots+d_s=e$.
\end{enumerate}
We shall refer to $f_1, \ldots, f_s$ as the \emph{polynomial
parameters} and to $\lambda_1,\ldots,\lambda_s$ as the \emph{
eigenvalues} of the foliation. If no confusion can arise we will
call this foliations just \emph{affine logarithmic foliations} or
\emph{logarithmic foliations} as well.
\end{Definition}

\

We will note $\LL(\C^n,\overline{d})$ the \emph{space of affine
logarithmic foliations} of this kind and define the \emph{generic}
open set $\mathcal{U}_{\LL_0}\subset \LL(\C^n,\overline{d})$ as
\begin{equation}\label{gen-log2}
\mathcal{U}_{\LL_0} = \left\{\omega\in \LL(\C^n,(\overline{d})):\
\omega\text{ verifies a) and b) below }\right\},
\end{equation}
writing $\omega=\left(\prod_{i=1}^s f_i \right)\sum_{i=1}^s\lambda_i
\frac{df_i}{f_i}$ we have the conditions:
\begin{enumerate}
\item[a)] $D = \{f_1.\ldots . f_s = 0\}$ is a normal crossing divisor
\item[b)] $\lambda_i \neq \lambda_j (\neq 0)$ for every $i\neq j$.
\end{enumerate}

\

Let us now consider $\omega_{\LL_0}\in\LL(\C^n,\overline{d})$ of the
form of eq. (\ref{logarithmic3}), then we define the subspaces of
$D_{\C^n}(\omega_{\LL_0},e)$ as
\[
\begin{aligned}
D_{\C^n}(\omega_{\LL_0},\overline{f}) &= Span\left(\{\eta_{g_i}\in\LL(\C^n,\overline{d}):\ \eta_{g_i}\text{ equals }\omega_{\LL_0}\text{ with }f_i\text{ changed by }g_i\}\right)\big/\C.\omega_{\LL_0}\\
D_{\C^n}(\omega_{\LL_0},\overline{\lambda}) &=
Span\left(\{\eta_{\overline{\mu}}\in\LL(\C^n,\overline{d}):\
\eta_{\overline{\mu}}\text{ equals }\omega_{\LL_0}\text{ with
}\overline{\lambda}\text{ changed by
}\overline{\mu}\}\right)\big/\C.\omega_{\LL_0}.
\end{aligned}
\]

\

Later, in Theorem (\ref{prop1}) we will see that, again, these two
spaces span all the first order deformations of the same degree of
$\omega_{\LL_0}$.

\

\begin{Remark}\label{rem-1}As before, we would like to notice that we
may also consider  the space
$D_{\C^n}(\omega_{\LL_0},\overline{f})^+$ defined as
\[
 D_{\C^n}(\omega_{\LL_0},\overline{f})^+ =
  Span\left(\{\eta_{g_i}\in\LL(\C^n,\overline{d'}):\
  \eta_{g_i}\text{ equals }\omega_{\LL_0}
  \text{ with }f_i\text{ changed by }g_i\}\right)\big/\C.\omega_{\LL_0}
\]
and $\overline{d'}$ is the $s$-uple defined as
$(d_1,\ldots,d_{i-1},d_i',d_{i+1},\ldots,d_s)$, where $d_i'$ is the
degree of the polyonomial $g_i$, which can be different from the
original degree $d_i$.
\end{Remark}

\

Again, by direct computation, it is straight forward to check that
$D_{\C^n}(\omega_{\LL_0},\overline{f})$ and
$D_{\C^n}(\omega_{\LL_0},\overline{\lambda})$ are subspaces of
$D_{\C^n}(\omega_{\LL_0},e)$ and that the space
$D_{\C^n}(\omega_{\LL_0},\overline{f})^+$ is a subspace of
$D_{\C^n}(\omega_{\LL_0})$.

\begin{Remark}
{\rm Notice that for $s=2$ an affine logarithmic foliation is also
an affine rational foliation.  }
\end{Remark}

\section{Affine deformations of  affine rational and logarithmic foliations}
\label{affinedef}

\

Along this section we are going to prove that an affine first order
deformation of a homogeneous differential form $\omega$ can be
projectivized and it still defines a first order deformation of the
projectivization of $\omega$, given that they are homogeneous and
have the same degree, see Lemma (\ref{affine-def}) below. This lemma
is used for proving Theorem (\ref{prop0}) and Theorem (\ref{prop1})
which classify first order perturbations of affine rational and
logarithmic foliations.

\

Let us consider a differential form $\eta\in \Omega^1_{\C^n}$,
homogeneous, of degree $e$. If $\eta = \sum h_i dx_i$, then its
\emph{projectivization} is given by
\begin{equation}\label{projectivization}
   \widetilde{\eta} = z\eta - i_R(\eta)dz = z\eta - \left(\sum x_i h_i\right)dz\ .
   \end{equation}
And we also have that
\[
  d\widetilde{\eta} = - 2 \eta\wedge dz - \left(\sum x_i dh_i\right)\wedge dz + z d\eta\ .
\]
We have the following equality, having $d\eta = \sum dh_i\wedge
dx_i$ we get that, since the degree of $\eta$ is equal to $e$, and
following Euler's formula $i_R(dh_i) = \partial(h_i)h_i$,
\begin{equation}\label{equ5}
 i_R(d\eta) = \sum (e-1)h_i dx_i - \sum x_i dh_i = (e-1) \eta - \sum x_i dh_i\ .
\end{equation}

\

\begin{Lemma}\label{affine-def}Let us consider $\omega$ and $\eta$ a
degree $e$ homogeneous,
 1-differential forms such that
\[
\omega\wedge d \omega=0, \,  \omega\wedge d\eta + d\omega\wedge \eta
= 0\ .
\]
Now, consider the projectivization of these two differential forms
in the sense of eq. (\ref{projectivization}), let us name them
$\widetilde{\omega}$ and $\widetilde{\eta}$, respectively.

Then, we have that
\[
 \widetilde{\omega}\wedge d\widetilde{\eta} +
 d\widetilde{\omega}\wedge \widetilde{\eta} = 0
\]
\end{Lemma}

\begin{proof}
From the above  equations we have the following equalities, writing
$\omega$ as $\omega=\sum_{i=1}^nf_idx_i$,
 \[
\begin{aligned}
  \widetilde{\omega} &= z\omega- \left(\sum x_i f_i\right)dz\\
  d\widetilde{\omega} &= - 2 \omega\wedge dz - \left(\sum x_i df_i\right)
  \wedge dz + z d\omega\\
  \widetilde{\eta} &= z\eta - \left(\sum x_i h_i\right)dz\\
  d\widetilde{\eta} &= - 2 \eta\wedge dz - \left(\sum x_i dh_i\right)
  \wedge dz + z d\eta\ .
\end{aligned}
\]
then to compute $\widetilde{\omega}\wedge d\widetilde{\eta} +
d\widetilde{\omega}\wedge \widetilde{\eta}$ we proceed as follows:
\[
 \begin{aligned}
  \widetilde{\omega}\wedge d\widetilde{\eta} &= -z \left(\sum x_i f_i\right)dz\wedge d\eta + z\omega\wedge\left[- 2 \eta\wedge dz - \left(\sum x_i dh_i\right)\wedge dz \right] +\\
  &\hspace{1cm} + z^2 \omega\wedge d\eta=\\
  &= -z \left(\sum x_i f_i\right)dz\wedge d\eta -2 z\omega\wedge\eta\wedge dz - z\omega\wedge \left(\sum x_i dh_i\right)\wedge dz 9\\
  &\hspace{1cm}+ z^2 \omega\wedge d\eta = \\
  &= -z \left(\sum x_i f_i\right) d\eta \wedge dz -2 z\omega\wedge\eta\wedge dz - z\omega\wedge \left(\sum x_i dh_i\right)\wedge dz +\\
  &\hspace{1cm} + z^2 \omega\wedge d\eta \\
  d\widetilde{\omega}\wedge \widetilde{\eta} &= z^2 d\omega\wedge \eta - z \left(\sum x_i h_i \right) d\omega\wedge dz - 2z\omega\wedge dz\wedge\eta +\\
  &\hspace{1cm}-z \left(\sum x_i df_i \right)\wedge dz\wedge \eta = \\
  &= z^2 d\omega\wedge \eta - z \left(\sum x_i h_i \right) d\omega\wedge dz + 2z\omega\wedge\eta \wedge dz+\\
  &\hspace{1cm}+z \left(\sum x_i df_i \right)\wedge \eta \wedge dz=
 \end{aligned}
\]

So, we finally get:
\[
 \begin{aligned}
  \widetilde{\omega}\wedge d\widetilde{\eta} & + d\widetilde{\omega}\wedge \widetilde{\eta} = -z \left(\sum x_i f_i\right) d\eta\wedge dz- z\omega\wedge \left(\sum x_i dh_i\right)\wedge dz +\\
  &\hspace{1cm}- z \left(\sum x_i h_i \right) d\omega\wedge dz+z \left(\sum x_i df_i \right)\wedge \eta \wedge dz= \\
  & = z \left[-\left(\sum x_i f_i\right) d\eta - \omega\wedge \left(\sum x_i dh_i\right) - \left(\sum x_i h_i\right) d\omega+\right.\\
&\hspace{1cm}  \left.+\left(\sum x_i df_i \right)\wedge \eta
\right]\wedge dz
 \end{aligned}
\]

This way, we would like to see the annihilation of the following
equation
\begin{equation}\label{equ4}
-\left(\sum x_i f_i\right) d\eta - \omega\wedge \left(\sum x_i
dh_i\right) - \left(\sum x_i h_i\right) d\omega+\left(\sum x_i df_i
\right)\wedge \eta= 0
\end{equation}

\

This can be seen by contracting the following equation
\[
 \omega\wedge d\eta + d\omega\wedge \eta = 0
\]
with the radial vector field $R$. Then we get that, following eq.
(\ref{equ5}),
\[
\begin{aligned}
  \left(\sum x_i f_i \right) d\eta &- \omega \wedge \left[(e-1)\eta- \left(\sum x_i dh_i\right)\right] + \left[(e-1)\omega- \left(\sum x_i df_i\right)\right]\wedge \eta +\\
  &\hspace{1cm}+ \left(\sum x_i h_i\right)d\omega = \\
  &=\left(\sum x_i f_i \right) d\eta + \omega\wedge \left(\sum x_i dh_i\right) - \left(\sum x_i df_i\right)\wedge \eta +\\
  &\hspace{1cm} +\left(\sum x_i h_i\right)d\omega = 0
\end{aligned}
\]
showing that eq. (\ref{equ4}) is zero, as expected.
\end{proof}

\begin{Remark}
 {\rm The converse of the above lemma is clear: if
 $\widetilde{\omega}$ and $\widetilde{\eta}$ are such that
 $ \widetilde{\omega}\wedge d\widetilde{\eta} + d\widetilde{\omega}\wedge \widetilde{\eta} =0
 $ then $ \omega\wedge d\eta+d\omega\wedge\eta = 0\ .$}

\end{Remark}

Now, let us see that  the projectivization of an affine rational and
logarithmic foliation given by
$\omega_{\RR_0}\in\RR(\C^n,\overline{d})$ and
$\omega_{\LL_0}\in\LL(\C^n,\overline{d})$, respectively, is
logarithmic. For that, let us write as $\omega_0$ to both of our
foliations, and since $i_R(\omega_{\RR_0})=i_R(\omega_{\LL_0}) = \mu
F$, where $\mu = d_1r-d_2s$ for the rational case and $\mu =
\sum_{i=1}^s \lambda_i d_i$ for the logarithmic case, we just need
to see that
\[
 \begin{aligned}
  \widetilde{\omega_{0}} &= z \omega_{0} - \left(\sum_{i=1}^n x_i i_{\frac{\partial}{\partial x_i}}(\omega_{0}) \right)dz  = z\omega_{0} - i_R(\omega_{0}) dz = \\
  &= z \omega_{0} - \mu F\ dz
 \end{aligned}
\]
which has effectively the form of a logarithmic foliation defined in
${\mathbb P}^n$, with parameters given by
\[
\begin{aligned}
  f_1,f_2 \text{ and }z \qquad &\text{and} \qquad d_1,d_2 \text{ and } (-\mu)&&\text{ for }\omega_0\text{ rational}\\
  f_1,\ldots,f_s\text{ and } z\qquad &\text{and}\qquad \lambda_1,\ldots,\lambda_s\text{ and }(-\mu)&&\text{ for }\omega_0\text{ logarithmic.}\\
\end{aligned}
\]

\

Now, by \cite[Theorem 25, pp.~14, and Remark 26, pp.~15]{cgm} we
know that the first order (same degree projective homogeneous)
deformations (of a projective homogeneous logarithmic foliation)
 are given by perturbing the polynomial parameters $f_i$ and $z$
and the eigenvalues $\lambda_i$ (or the $d_1,d_2$) and $\mu$.

After dehomogenization we get that the only perturbations of the
same degree of $\omega_0$, \emph{i.e.} of degree $e$, are those
given by the perturbations of the polynomial parameters $f_i$ and
eigenvalues $\lambda_i$ (or the $d_1,d_2$), since the perturbation
given in the direction of the infinite hyperplane $z$, after
dehomogenization, gives a differential form of degree $e+1$, and the
perturbation of $\mu$, after dehomogenization, gives the trivial
deformation. Using the fact that the first order (degree one in the
parameter) restricts the perturbation to either the polynomial
parameters or the eigenvalues we obtain:



\begin{Theorem}\label{prop0}
 Let $\omega_{\RR_0}\in {\mathcal U}_{\RR_0}$ be a generic affine rational foliation
 in $\RR(\C^n,\overline{d})$, defined by homogeneous polynomials $f_1,f_2$
 \[
\omega_{\RR_0} = r\ f_1df_2\ -\ s\ f_2df_2.
 \]
Then all first order perturbations
 of $\omega_{\RR_0}$, of the same degree of $\omega_{\RR_0}$,
 are the perturbations of the polynomial parameters $f_1$ and $f_2$
 or of the eigenvalues $(r,s)$, \emph{i.e.}, we have that
 \[
  D_{\C^n}(\omega_{\RR_0},e) = D_{\C^n}(\omega_{\RR_0},
  \overline{f})\oplus D_{\C^n}(\omega_{\RR_0},(r,s))
 \]
\end{Theorem}

\

\begin{Theorem}\label{prop1}
 Let $\omega_{\LL_0}\in {\mathcal U}_\LL(d)$ be a generic affine logarithmic foliation
 in $\LL(\C^n,\overline{d})$, defined by homogeneous polynomials $f_1,\ldots,f_s$
 \[
\omega_{\LL_0} = \left(\prod_{k=1}^s f_k \right)\left(\sum_{k=1}^s
\lambda_k \frac{df_k}{f_k}\right) = \sum_{k=1}^s \lambda_k\  F_k
df_k\, .
 \]
Then all first order perturbations
 of $\omega_{\LL_0}$, of the same degree of $\omega_{\LL_0}$,
 are the perturbations of the polynomial parameters $f_i$
 or of the eigenvalues $\lambda_i$, \emph{i.e.}, we have that
 \[
  D_{\C^n}(\omega_{\LL_0},e) = D_{\C^n}(\omega_{\LL_0},
  \overline{f})\oplus D_{\C^n}(\omega_{\LL_0},\overline{\lambda})
 \]

\end{Theorem}

\

\section{Relative cohomology with poles:
the equation $d\left(\frac{\eta}{F}\right)\wedge \omega_{0} = 0$}

\ The problem of relative cohomology for holomorphic differential
forms has been studied by Cerveau and Berthier.  We recall that
given a one-form $\omega$ and a one-form $\eta$ both defined in the
same domain, we say that $\eta$ is closed relatively to $\omega$ if
$d \eta \wedge \omega=0$. We also say that $\eta$ is exact
relatively to $\omega$ if  $\eta = dh + a \omega$ for some
holomorphic functions $a$ and $h$ in the same domain of definition
as $\omega$ and $\eta$. The basic question is whether a relatively
closed one-form $\eta$ with respect to $\omega$ is also exact with
respect to $\omega$. Assume that the form $\omega$ is integrable,
\emph{i.e.}, $\omega \wedge d \omega=0$. In this case $\omega=0$
defines a holomorphic foliation of codimension one and with singular
set given by $\sing(\omega)$. In this case the condition $d\eta
\wedge \omega=0$ means that the restriction of $\eta$ to the leaves
of $\omega$ is a closed one-form. This indicates that the topology
of the leaves of $\omega$ may be an ingredient in the solution to be
the above question. In the case of germs of one-forms, Cerveau and
Berthier have proved (see \cite[Th\'eor\`em 4.1.1, pp.~422]{CB})
that for a generic logarithmic one-form $\omega_{\LL_0}$ with some
additional diophantine conditions in the coefficients $\lambda_j$,
the equation
\[
 d\eta \wedge \omega_{\LL_0}=0
\]
is equivalent to the fact that  $\eta$ is of the form
\[
 \eta = a \omega_{\LL_0} + dh
\]
for some $a,h\in\O_{\C^n,0}$, \emph{i.e.}, $a,h$ germs of
holomorphic functions in $n$-variables around $0\in\C^n$.

Nevertheless, we shall address a different situation. Since we are
interested in deformations of an affine rational or logarithmic
foliation, we must study the following equation
\begin{equation}\label{equ1}
 d\left(\frac{\eta}{F}\right)\wedge \omega_{0} = 0
\end{equation}
where $\omega_{0}\in\RR(\C^n,\overline{d})$ or
$\omega_0\in\LL(\C^n,\overline{d})$ is an affine rational or
logarithmic foliation of type
\begin{equation}\label{rational1}
  \omega_0 = r\ f_1df_2 \ - \ s \ f_2df_1
\end{equation}
in the rational case, or
\begin{equation}\label{logarithmic4}
  \omega_{0} = \left(\prod_{k=1}^s f_k \right)
  \left(\sum_{k=1}^s \lambda_k \frac{df_k}{f_k}\right) =
  \sum_{k=1}^s \lambda_k \ F_k df_k
\end{equation}
where $F_k = \prod_{j\neq k} f_j$, $\lambda_k\in\C$ and
$F=\prod_{k=1}^s f_k$, in the logarithmic case. In other words, we
would like to know what happens when we divide $\eta$ by the
polynomial $F$, the integrating factor of the affine rational or
logarithmic differential form $\omega_{0}$.

\

In any case, we are going to work with the following equivalent
equation to eq. (\ref{equ1}) which is
\begin{equation}\label{equ2}
(Fd\eta - dF\wedge \eta)\wedge \omega_{0} = 0\ ,
\end{equation}
and we will still denote with $F$ the product $f_1.f_2$, of the
polynomials involved in the defintion of the affine rational
foliation.

\

\subsection{First order perturbations (solutions of degree $=\partial(\omega_{0})$)}

\

\

Along this section we are going to see the equivalence between the
equation that defines a first order deformation of a foliation
defined by an affine rational or logarithmic form $\omega_0$, which
says that it is given by the differential forms $\eta$ such that,
see eq. (\ref{equ0}),
\begin{equation}\label{equ6}
\omega_0\wedge d\eta + d\omega_0\wedge \eta = 0\ .
\end{equation}
and between the equation defining the relative cohomology of
$\omega_0$, with poles in the integrating factor defined by
$\omega_0$, see eq. (\ref{equ2}),
\[
 d\left(\frac{\eta}{F}\right)\wedge \omega_0 = 0\ ,
\]
see Corollary (\ref{coro1}) below for a complete statement of the
result.

\

We begin with the following proposition:
\begin{Proposition}\label{prop2} If $\omega_{0}\in\RR(\C^n,\overline{d})$ or $\omega_0\in\LL(\C^n,\overline{d})$ then if $\eta\in\Omega^1_{\C^n}$ then we have that
\[
\omega_{0}\wedge d\eta + d\omega_{0}\wedge \eta = 0 \Rightarrow
(Fd\eta - dF\wedge \eta)\wedge \omega_{0} = 0
\]
\end{Proposition}
\begin{proof}
For this we make use of the following  well-known fact. For
$\omega_{0}$ rational, as in Definition (\ref{rat-affine}), or
logarithmic, as in Definition (\ref{log-affine}), the following
equation holds, since $F$, defined as above, is an integrating
factor of $\omega_{0}$ then
\begin{equation}\label{equ3}
  Fd\omega_{0} = dF\wedge \omega_{0}\ .
\end{equation}

Then, by multiplying by $F$ the eq. (\ref{equ6}) we get
\[
 \begin{aligned}
  F\omega_{0}\wedge d\eta + F d\omega_{0}\wedge \eta &= 0 \\
  F \omega_{0}\wedge d\eta + dF\wedge \omega_{0}\wedge \eta &= 0\\
  Fd\eta\wedge \omega_{0} - dF\wedge \eta\wedge \omega_{0} &= 0\\
 (Fd\eta - dF\wedge \eta )\wedge \omega_{0} &= 0
\end{aligned}
\]
concluding our first result.
\end{proof}

\

\begin{Proposition} If $\omega_{0}\in\RR(\C^n,\overline{d})$ or $\omega_0\in\LL(\C^n,\overline{d})$
and if $\eta\in\Omega^1_{\C^n}$ is homogeneous and has the same
degree of $\omega_{0}$ then we have that
\[
\omega_{0}\wedge d\eta + d\omega_{0}\wedge \eta = 0 \Leftarrow
(Fd\eta - dF\wedge \eta)\wedge \omega_{0} = 0
\]
\end{Proposition}
\begin{proof}
 To see this, let us proceed as follows.

\

First we apply the exterior diferential to eq. (\ref{equ2}), getting
\[
\begin{aligned}
 & 2dF\wedge d\eta\wedge\omega_{0} +
 (F\ d\eta - dF\wedge \eta)\wedge \omega_{0} = 0,
\end{aligned}
\]
now, we multiply the above equation by $F$, and using eq.
(\ref{equ3}) above, we get
\[
 \begin{aligned}
   2F\ dF\wedge d\eta\wedge\omega_{0} +
   F\ (F\ d\eta - dF\wedge \eta)\wedge d\omega_{0} &= 0\\
   2F\ dF\wedge d\eta\wedge\omega_{0} + (F\ d\eta - dF\wedge \eta)
   \wedge dF\wedge\omega_{0} &= 0\\
   2F\ dF\wedge d\eta\wedge\omega_{0} + F\ d\eta \wedge
   dF\wedge\omega_{0} &= 0\\
   2F\ dF\wedge d\eta\wedge\omega_{0} + F\ dF\wedge
    d\eta \wedge \omega_{0} &= 0\\
   3F\ dF\wedge d\eta\wedge\omega_{0} &=0\\
  dF\wedge d\eta\wedge\omega_{0} & =0\\
  d\eta\wedge dF\wedge \omega_{0} & =0\\
  F \ d\eta\wedge d\omega_{0} &= 0\\
  d\eta\wedge d\omega_{0} &= 0\\
 \end{aligned}
\]

Now, we use the contraction with the radial vector field $R$ applied
to the last equation, and using Cartan's formula
\[
 L_R(\omega_{0}) = e\omega_{0} = di_R(\omega_{0}) + i_R(d\omega_{0}) ,
\]
we get
\[
 \begin{aligned}
 i_R(d\eta)\wedge d\omega_{0} + d\eta\wedge i_R(d\omega_{0}) &= 0\\
  e\eta\wedge d\omega_{0} - di_R(\eta)\wedge d\omega_{0}+
ed\eta\wedge \omega_{0} -d\eta\wedge di_R(\omega_{0})&= 0\\
 \end{aligned}
\]
wich can be written as
\[
 \omega_{0}\wedge d\eta + d\omega_{0}\wedge \eta =
 \frac{1}{e}\left[di_R(\eta)\wedge d\omega_{0} +
 d\eta\wedge di_R(\omega_{0})\right]
\]

where we are assuming that $\partial(\omega_{0})=\partial(\eta) =
e$.

\

Now, let us see that the right side of this last equation is zero.
So, we want to see that
\begin{equation}\label{equ7}
 di_R(\eta)\wedge d\omega_{0} + d\eta\wedge di_R(\omega_{0}) = 0\ .
\end{equation}

\

For that, we are going to apply the contraction with the radial
vector field to eq. (\ref{equ2}), and we will also write
$i_R(\omega_{0}) = \mu F$, for a $\mu\in\C$. We get:
\[
 \begin{aligned}
 \left[F i_R(d\eta) - eF\eta+i_R(\eta) dF\right]\wedge \omega_{0} +
 \mu F\left(Fd\eta-dF\wedge\eta\right)&= 0\\
F\left[\underset{=e\eta-di_R(\eta)}{i_R(d\eta)\wedge\omega_{0}} -
e\eta\wedge \omega_{0} +i_R(\eta) d\omega_{0} +
\mu (Fd\eta-dF\wedge\eta)\right] &= 0\\
F\left[\omega_{0}\wedge di_R(\eta) +i_R(\eta) d\omega_{0} +
\mu (Fd\eta-dF\wedge\eta)\right] &= 0\\
\omega_{0}\wedge di_R(\eta) +i_R(\eta) d\omega_{0} +
\mu (Fd\eta-dF\wedge\eta) &= 0\\
\end{aligned}
\]
This last equation, can be rewritten as
\[
 \omega_{0}\wedge di_R(\eta) + i_R(\eta) d\omega_{0}  =
  - \mu (Fd\eta-dF\wedge\eta)\ .
\]
Applying the exterior differential to this last equation we get
\begin{equation}
 \begin{aligned}\label{equ8}
  2\ di_R(\eta) \wedge d\omega_{0}  &= - 2\mu \ dF\wedge d\eta\\
  di_R(\eta)\wedge d\omega_{0} &= - \mu \ dF\wedge d\eta \\
 \end{aligned}
\end{equation}

\

If, we now take eq. (\ref{equ7}) and we use the equality
$i_R(\omega_{0}) = \mu F$ then we have
\[
 di_R(\eta)\wedge d\omega_{0} + \mu\ d\eta\wedge dF \ ,
\]
and using eq. (\ref{equ8}) we finally get that
\[
 - \mu \ dF\wedge d\eta + \mu\ d\eta\wedge dF = 0
\]
as we wanted to see.
\end{proof}

\

\begin{Corollary}\label{coro1}
If $\omega_{0}\in\RR(\C^n,\overline{d})$ or
$\omega_0\in\LL(\C^n,\overline{d})$ then if $\eta\in\Omega^1_{\C^n}$
is homogeneous and has the same degree of $\omega_{0}$, let us
assume that $\partial(\omega_{0})=e$, then
\[
\eta\in D(\omega_{0},e) \iff d\left(\frac{\eta}{F}\right)\wedge
\omega_{0} = 0
\]
\end{Corollary}

\

\begin{Remark}
 With this result, we have that if $\eta$ is a perturbation of
  $\omega_{0}$ given by changing one of the parameters $f_i$
  or one of the $\lambda_i$ (or one of $d_1,d_2$), then it verifies eq. (\ref{equ2}).

Following \cite{cgm} and the computations of section
(\ref{affinedef}), see Theorem (\ref{prop0}) and Theorem
(\ref{prop1}), we have that this are all possible solutions for
an affine rational or logarithmic foliation, if the $\omega_{0}$ is
generic and we consider only solutions of the same degree of $\omega_0$.
\end{Remark}

\

\subsection{Solutions of degree $\neq\partial (\omega _{0})$}

\

\

In this section we show some examples of solutions of the equation
of relative cohomology, see eq. (\ref{equ2}), by using Proposition
(\ref{prop2}), when considering degrees of $\eta$ such that
$\partial(\eta)\neq \partial (\omega_0)$.

\

By following Remark (\ref{rem-1}), considering
$D_{\C^n}(\omega_{0},\overline{f})^+$ where
$\omega_{0}\in\RR(\C^n,\overline{d})$ or
$\omega_0\in\LL(\C^n,\overline{d})$, we have that as a corollary of
Proposition (\ref{prop2}) above, we can also get solutions of the
eq. (\ref{equ2}) of different degrees to the one given by
$\omega_{0}$. In contrast to the case when the degree is the same as
of the original $\omega_{0}$, as we show in the preceeding section,
we do not know whether these are all the possible solutions.

\

In particular we can give, as an example for the case when
$\omega_{0}\in\RR(\C^n,\overline{d})$ or
$\omega_0\in\LL(\C^n,\overline{d})$, the extreme case where the
polynomial $f_i$ is changed by the constant polynomial equal to 1,
and then we get that the differential form, in the logarithmic case,
\[
\eta = \sum_{j\in J} \lambda_j \overline{F}_j df_j
\]
is a solution of eq. (\ref{equ2}), for $J\subset [1,\ldots, s]$,
such that $\#(J)=s-1$ and $\overline{F}_j = \prod_{\substack{i\in J
\\i\neq j}}f_i$.

In the rational case, the situation is much simpler, since we should
consider $\eta$ such that
\[
 \eta = df_1 \qquad\text{or}\qquad \eta= df_2\ .
\]

\

\

\section{Deformations of dicritical homogeneous one-forms}

Let $\omega$ be a homogeneous one-form in $\mathbb C^n, n \geq 3$
satisfying the integrability condition $\omega \wedge d \omega=0$.
According to  \cite[Part 4, Chap. I pp. 86-95]{Cerveau-Mattei} we
have that either $F=i_R(\omega)\equiv 0$ or $F$ is an integrating
factor for $\omega$. In the non-dicritical case, \emph{i.e.}, for $F
\not\equiv 0$, we can the write
\[
\frac{1}{F}\omega= \sum \limits_{i=1}^s \lambda _i \frac{df_i}{f_i}
+ d\left(\frac{g}{\prod\limits_{i=1}^s f_i^{n_i -1}}\right)
\]
for some $\lambda_i \in \mathbb C, n_i \geq 1$ and some homogeneous
polynomials $f_i, g$. Clearly the $f_i$ are factors of $F$ so that
we must have $F=\prod_{i=1}^s f_i^{n_i}$. Put $f=f_1\ldots f_s$. We
may then rewrite
\[
\omega= \sum\limits_{i=1}^s \lambda_i \ \frac{F}{f_i}df_i + fdg -
g\sum\limits_{i=1}^s (n_i -1)\frac{f}{f_i} df_i
\]
We may deform $\omega$ as follows:
\[
\omega_ \varepsilon = F\left[\sum \limits_{i=1}^s \lambda _i
\frac{df_i}{f_i} + (1 + \varepsilon)
d\left(\frac{g}{\prod\limits_{i=1}^s f_i^{n_i -1}}\right) \right]
\]
Notice that each $\omega_\varepsilon$ admits $F$ as integrating
factor and therefore it is integrable. We have $\omega_\varepsilon =
\omega  + \varepsilon \eta$ where $\eta = F
d\left(\frac{g}{\prod\limits_{i=1}^s f_i^{n_i -1}}\right)$. It is
interesting to observe that $F$ is also an integrating factor for
$\eta$, \emph{i.e.}, $\frac{1}{F}\eta$ is closed. The deformations
of non-dicritical homogeneous integrable one-forms (by same degree
homogeneous integrable one-forms) are described in \cite[Part 4,
Chap. I pp. 86-95]{Cerveau-Mattei}. Now we turn our attention to the
dicritical case, \emph{i.e.}, when $i_R(\omega) \equiv 0$. We have:

\begin{Proposition}
\label{Proposition:integratingfactor}
 Let $\omega\in\FF^1(\P^n,e)$ be a dicritical
 homogeneous one-form, \emph{i.e.},
 such that $i_R(\omega)=0$. Let $\eta$ be a solution of
 \[
(Fd\eta - dF\wedge \eta)\wedge \omega = 0
 \]
homogeneous of the same degree of $\omega$. Then either
$i_R(\eta)=0$ or it is an integrating factor of $\omega$ and $\eta$.
\end{Proposition}

\begin{proof}
Assume that $i_R(\eta) \not \equiv 0$. The one-forms
$\omega_\varepsilon= \omega + \varepsilon\eta$ are integrable and
homogeneous. Moreover, $i_R(\omega_\varepsilon) = i_R(\omega) +
\varepsilon i_R(\eta)= \varepsilon i_R(\eta)\neq 0, \forall
\varepsilon \ne 0$. According to \cite{Cerveau-Mattei} as mentioned
above, the one-form
$\frac{1}{i_R(\omega_\varepsilon)}\omega_\varepsilon$ is closed for
all $\varepsilon \ne 0$. Let $F=i_R(\eta)$. We have
\[
\frac{1}{i_R(\omega_\varepsilon)}\omega_\varepsilon=
\frac{1}{\varepsilon F}(\omega + \varepsilon \eta) =
(\varepsilon)^{-1}\frac{1}{F} \omega + \frac{1}{F}\eta.
\]
This implies that $\frac{1}{F} \omega$ and $\frac{1}{F}\eta$ are
closed.

\end{proof}

\

\begin{Corollary}
 Let $\omega\in\FF^1(\P^n,e)$ be a dicritical
 homogeneous one-form, \emph{i.e.},
 such that $i_R(\omega)=0$. Given a first order deformation
 $\omega_\varepsilon = \omega + \varepsilon \eta$  of $\omega$ by integrable homogeneous one-forms we have the
 following possibilities:
 \begin{enumerate}
 \item $\omega_\varepsilon$ descends to the projective space $\mathbb
 P^{n}$.

 \item $\omega_\varepsilon$ is of the form
 \[
 \omega_\varepsilon =  \left(\prod \limits_{i=1} ^s f_i ^{n_i}\right)\left[ \sum\limits_{i=1} ^s (\lambda_i + \varepsilon
 \mu_i) \frac{df_i}{f_i} + d\left(\frac{ g + \varepsilon h}{\prod\limits_{i=1}^s f_i
 ^{n_i -1}}\right)\right]
\]
where $f_i, g, h$ are homogeneous polynomials, $\lambda_i, \mu_i \in
\mathbb C$.
 \end{enumerate}
\end{Corollary}
\begin{proof}
Let $F= i_R(\eta)$. If $F=0$ then $i_R(\omega_\varepsilon)=0$ and
therefore the deformation descends to the projective space $\mathbb
P^{n}$. Assume now that $F \ne 0$. From
Proposition~\ref{Proposition:integratingfactor} we know that
 $\frac{1}{F}\omega$ and $\frac{1}{F}\eta$ are closed. Put $F=\prod\limits_{i=1}^s f_i ^{n_i}$ in
 irreducible homogeneous distinct factors. Then we can
 apply the Integration lemma from \cite{Cerveau-Mattei} in order to
 write
 \[
 \omega= F\left[\sum\limits_{i=1} ^s \lambda_i  \frac{df_i}{f_i} +
 d\left(\frac{ g}{\prod\limits_{i=1}^s f_i
 ^{n_i -1}}\right)\right]
\]
and
\[
\eta= F\left[\sum\limits_{i=1} ^s  \mu_i \frac{df_i}{f_i} +
d\left(\dfrac{ h}{\prod\limits_{i=1}^s f_i
 ^{n_i -1}}\right)\right]
\]
Then the result follows.

\end{proof}

\begin{Remark}
{\rm In  case (1) the deformation can be viewed in the projective
space $\mathbb P^{n}$ and then we can apply the above discussion and
corollary once again. }
\end{Remark}

\

\section{Stability of an exact differential form $\omega=dP$}

\

Along this section we would like to study the stability under
perturbations of an exact differential form of type $\omega=dP$,
where $P\in S_n$ is a polynomial of degree $e$. As in the former
sections we are considering first order deformations. These are
given by one-forms $\omega_\varepsilon= \omega + \varepsilon \eta$,
where $\eta$ is homogeneous of degree $e$ and each
$\omega_\varepsilon$ is integrable $\omega_\varepsilon \wedge d
\omega_ \varepsilon =0$.

In  \cite{cs} the authors prove a more general result than the one
in Theorem (\ref{teo-exact}). Anyway, we are writing this result
here because of its similarity with the previous method of
demonstration. Moreover, we highlight its strictly algebraic
character, unlike the techniques used in \cite{cs}.

\

Le us consider $\omega\in\Omega^1_{\C^n}$ of the form
\[
 \omega = dP
\]
for a homogeneous polynomial $P\in S_n$ of degree $e$. As before, we
are going to consider only deformations of the same degree of
$\omega$.

We would like to prove the following statement:

\begin{Theorem}\label{teo-exact}
 Let $\omega=dP$ be an exact differential form
 in $\Omega^1_{\C^n}$, homogeneous of degree $e$. Let us suppose
 also that the codimension of the singular locus of $dP$ is $\geq 3$.
 Then all first order deformations of $dP$, of the same degree, are of type
 \[
\omega_\varepsilon= d(P + \varepsilon Q)
 \]
where $Q$ is a homogeneous polynomial of degree $e$.
\end{Theorem}
\begin{proof}
Recalling Lemma (\ref{affine-def}), we consider the projectivization
of such a differential form $\omega$. This way we get, by using
Euler's formula,
\begin{equation}\label{equ10}
 \widetilde{\omega} = z \omega - i_{R}(\omega) dz =
 zdP-\left(\sum_{i=1}^nx_i i_{\frac{\partial}{\partial x_i}}dP \right)dz = zdP-e P dz
\end{equation}
wich is a rational foliation of type $(1,e)$.

\

And, by the hypothesis on $P$ we have that
$\widetilde{\omega}\in\mathcal{U}_\RR$, then using Theorem
(\ref{teo-rat}) we know that the first order deformations of such a
foliation are the deformations of its polynomial parameters. Then we
have that the space of first order deformations of
$\widetilde{\omega}$ are given by
\[
 \widetilde{\eta}_1 = zdQ - eQ dz\qquad \text{and}\qquad \widetilde{\eta}_2 = ldP- ePdl
\]
where $Q$ is a homogeneous polynomial of degree $e$, and $l$ is an
homogeneous polynomial of degree 1.

\

Now, after de-homogenisation, we get in the first case
\[
 \eta_1 = dQ\ .
\]

In the second case, after dehomogenization we get a differential
form degree $e+1$, which we are not allowed to consider.
 Then we conclude that the deformation $\omega_\varepsilon= \omega+ \varepsilon \eta$ is exact
 of the form $\omega=dP_\varepsilon$ where $P_\varepsilon$ is homogeneous and such that $P_0=P$.
 The fact that the deformation is of order one implies that $P_\varepsilon = P + \varepsilon Q$ as stated.

\end{proof}


\begin{Remark}
We would like to clarify the meaning of the term \emph{generic} on
the polynomial $P$. The codimension of the singular locus of
$\widetilde{\omega}$ being $\geq2$ means nothing since $P$ and $z$
are always transversal, but the condition of codimension of the
singular locus of $dP\wedge dz$ being $\geq 3$, means, in particular
that $P$ has to be reduced, irreducible and smooth in codimension
$2$.
\end{Remark}

\def\cprime{$'$}
\providecommand{\bysame}{\leavevmode\hbox
to3em{\hrulefill}\thinspace}
\providecommand{\MR}{\relax\ifhmode\unskip\space\fi MR }
\providecommand{\MRhref}[2]{%
  \href{http://www.ams.org/mathscinet-getitem?mr=#1}{#2}
} \providecommand{\href}[2]{#2}

\vglue.2in

\begin{tabular}{l l}
Ariel Molinuevo &Bruno Sc\'ardua\\
Instituto de Matem\'atica & Instituto de Matem\'atica\\
Universidade Federal do Rio de Janeiro & Universidade Federal do Rio de Janeiro\\
Caixa Postal 68530 & Caixa Postal 68530\\
CEP. 21945-970 Rio de Janeiro - RJ & CEP. 21945-970 Rio de Janeiro - RJ\\
BRASIL & BRASIL
\end{tabular}

\end{document}